\documentclass[11pt,letterpaper]{amsart}
\usepackage{geometry}
\usepackage[authoryear,sort&compress]{natbib}
\bibpunct{[}{]}{;}{n}{,}{,}

\usepackage{amsmath}
\usepackage{amsfonts}
\usepackage{amssymb}
\usepackage{amsaddr}
\usepackage{booktabs}
\usepackage{algorithm}
\usepackage{comment}
\usepackage[shortlabels]{enumitem}
\usepackage{mathrsfs}
\usepackage{amsthm}
\usepackage{tikz}
    \usetikzlibrary{matrix}
    \usetikzlibrary{decorations.markings}
    \usetikzlibrary{arrows}
\usepackage{hyperref}
\usepackage[T1]{fontenc}
\usepackage{oldgerm}
\usepackage{scalerel,stackengine}
\usepackage{stmaryrd}
\usepackage{xcolor}
\usepackage{xfrac}
\usepackage{MnSymbol}

\newtheorem{theorem}{Theorem}[section]
\newtheorem{definition}{Definition}[section]

\begin{document}

\title{Variational Collision Integrators for Nonholonomic Lagrangian Systems}

\author{\'Alvaro Rodr\'iguez Abella}
\email{rodriguezabella@g.ucla.edu}
\address{Electrical and Computer Engineering Department, University of California, Los Angeles, CA 90095 USA.}
\author{Leonardo Colombo}
\email{leonardo.colombo@car.upm-csic.es}
\address{Centre for Automation and Robotics (CSIC-UPM), Ctra. M300 Campo Real, Km 0,200, Arganda del Rey - 28500 Madrid, Spain.}  \thanks{LC acknowledges financial support from Grant PID2022-137909NB-C21 funded by MCIN/AEI/ 10.13039/501100011033.}%

\begin{abstract}
A discrete theory for implicit nonholonomic Lagrangian systems undergoing elastic collisions is developed. It is based on the discrete Lagrange--d'Alembert--Pontryagin variational principle and the dynamical equations thus obtained are the discrete nonholonomic implicit Euler--Lagrange equations together with the discrete conditions for the elastic impact. To illustrate the theory, variational integrators with collisions are built for several examples, including a bouncing ellipse and a nonholonomic spherical pendulum evolving inside a cylinder.
\end{abstract}

\maketitle

\section{Introduction}

The Lagrange--Dirac formulation of mechanics introduced in \cite{YoMa2006a,YoMa2006b} constitutes a unifying framework for implicit systems with nonholonomic constraints. The variational approach is based on the Lagrange--d'Alembert--Pontryagin principle, an extension of the Hamilton principle to account for degenerate (implicit) systems and nonholonomic constrains. The versatility of this framework has made it suitable to perform reduction by symmetries \cite{GaYo2015,YoMa2007} and interconnection \cite{JaYo2014,Ro2023}, as well as to extend it to the infinite-dimensional setting \cite{GaRoYo2025,RoGaYo2023}. In addition, in \cite{RoCo2023,RoCo2024} it was extended to systems with elastic collisions by following the approach in \cite{FeMaOrWe2003} for unconstrained systems. By regarding external forces as control inputs, implicit port-controlled Lagrangian systems were introduced in \cite{YoMa2012}. From that perspective, the space of inputs is thus a subbundle of the cotangent bundle of the configuration space of the system.

The discretization of continuous theories is a ubiquitous tool in geometric mechanics to build variational integrators \cite{MaWe2001}. The discretization of the Lagrange--Dirac mechanics was carried out in \cite{LeOh2010,LeOh2011} (see also \cite{CaFeToZu2023,PeYo2024} for alternative approaches), leading to discrete interconnection \cite{PaLe2016} and reduction \cite{RoLe2023}.

The aim of this work is to build a discrete analogous of the Lagrangian theory for implicit nonholonomic systems with elastic collisions introduced in \cite{RoCo2023,RoCo2024}, thus obtaining variational integrators for such systems. In order to solve the collision, we follow the approach in \cite{FeMaOrWe2003}, where the impact time is computed as $\tilde t=t_i+\alpha h$, where $h$ is the time-step, $t_i$ is the previous discrete time and $\alpha\in(0,1)$ is a variational variable of the discrete problem that has to be determined. The application of these ideas to nonholonomic implicit systems is not straightforward, as a projection from the tangent bundle of the configuration bundle to the tangent bundle of its boundary has to be used in order to couple the generalized momenta at the impact time with the velocity at the previous discrete time. Once the discrete Hamilton--d'Alembert--Pontryagin variational principle is established, the discrete equations follow from the discrete Lagrange--d'Alembert--Pontryagin principle, and are illustrated by simulating a bouncing ellipse and a nonholonomic spherical pendulum evolving inside a cylinder. Although it is not presented here, we aim to extend our work to include control inputs following the approach in \cite{YoMa2012}. Namely, external forces are regarded as inputs and added to the system. In addition, the controller may be computed by minimizing some cost functional, yielding optimal control equations \cite{CoMa2013}.

\section{Preliminaries}

In this paper, the nonsmooth variational approach introduced in \cite{RoCo2023,RoCo2024} to implicit nonholonomic Lagrangian systems with collisions is discretized following an analogous scheme to the one developed in \cite{FeMaOrWe2003} for unconstrained systems. The treatment of nonholonomic constraints is based on the discrete Lagrange--d'Alembert--Pontryagin principle introduced in \cite{LeOh2010,LeOh2011}. Let us briefly recall this approach, which utilizes retractions to discretize the constrain distribution.

\subsection{Retractions and discretization maps}\label{sec:retraction}

A \emph{retraction} (cf. \cite[Definition 9.1]{LeOh2011}) on a smooth manifold $Q$ is a map $\mathcal R:TQ\to Q$ such that
\begin{enumerate}
    \item $\mathcal R_q(0_q)=q$ for each $q\in Q$, where $\mathcal R_q=\mathcal R|_{T_qQ}:T_q Q\to Q$, and
    \item $T_{0_q}\mathcal R_q={\rm id}_{T_q Q}$ via the identification $T_{0_q}(T_q Q)\simeq T_q Q$.
\end{enumerate}

Given a retraction $\mathcal R:TQ\to Q$, we define the following map:
\begin{equation*}
\tilde{\mathcal R}:TQ\to Q\times Q,\quad v_q\mapsto\tilde{\mathcal R}_d(v_q)=(q,\mathcal R_q(v_q)).
\end{equation*}
Note that $\tilde{\mathcal R}(0_Q)=\Delta_Q$, where we denote $0_Q=\{0_q\in T_q Q\mid q\in Q\}$ and $\Delta_Q=\{(q,q)\in Q\times Q\mid q\in Q\}$. The fact that $T_{0_q}\mathcal R_q={\rm id}_{T_q Q}$ ensures that $\mathcal R_q:T_qQ\to Q$ is invertible around $0_q\in T_q Q$ for each $q\in Q$. Consequently, $\tilde{\mathcal R}$ is also invertible around $0_Q$, i.e., there exist $\mathcal U_{0_Q}\subset TQ$ neighborhood of $0_Q$ and $\mathcal U_{\Delta_Q}\subset Q\times Q$ neighborhood of the diagonal $\Delta_Q$, such that $\tilde{\mathcal R}|_{\mathcal U_{0_Q}}:\mathcal U_{0_Q}\to\mathcal U_{\Delta_Q}$ is invertible. The corresponding inverse defines a \emph{discretization map} (cf., for example, \cite{BaMa2022}),
\begin{equation*}
\Psi_d=\tilde{\mathcal R}|_{\mathcal U_{0_Q}}^{-1}:\mathcal U_{\Delta_Q}\to\mathcal U_{0_Q}.
\end{equation*}
For simplicity, we will write $\Psi_d:Q\times Q\to TQ$ understanding that this map is only locally defined around the diagonal $\Delta_Q$. 

Given a (possibly degenerate) continuous Lagrangian $L:TQ\to\mathbb R$, the corresponding discrete Lagrangian is defined as $L_d=L\circ\Psi_d:Q\times Q\to\mathbb R$. Typically, the discretization map and, thus, the discrete Lagrangian, depends on the timestep $h\in(0,1)$. This will be explicitly denoted by $L_d:Q\times Q\times(0,1)\to\mathbb R$.

\subsection{Discrete nonholonomic constraints}\label{sec:discreteconstraint}

Let $\Delta_Q\subset TQ$ be a (possibly nonholonomic) constraint distribution on a smooth manifold $Q$, and $\Delta_Q^\circ\subset T^*Q$ be its annihilator. By denoting $m=\dim Q-\dim\Delta_Q$, there exist $\omega^\mu\in\Omega^1(Q)$, $1\leq\mu\leq m$, such that
\begin{align*}
\Delta_Q^\circ(q)={\rm span}\{\omega_q^\mu\in T_q^*Q\mid1\leq\mu\leq m\},\qquad q\in Q.
\end{align*}
In order to discretize the constraint distribution, we utilize a retraction $\mathcal R:TQ\to Q$ as in \S\ref{sec:retraction}. Specifically, for each $1\leq\mu\leq m$, we introduce the discrete maps (cf. \cite[\S3]{CoMa2001})
\begin{align*}
\omega_{d+}^\mu:Q\times Q\to\mathbb R,\quad & (q,v)\mapsto\omega_{d+}^\mu(q,v)=\omega_q^\mu(\Psi_d(q,v))=\omega_q^\mu(\mathcal R_q^{-1}(v)),\\
\omega_{d-}^\mu:Q\times Q\to\mathbb R,\quad & (v,q)\mapsto\omega_{d-}^\mu(v,q)=-\omega_q^\mu(\Psi_d(q,v))=-\omega_q^\mu(\mathcal R_q^{-1}(v)).
\end{align*}
Therefore, the \emph{discrete constraint distributions} are defined as
\begin{align*}
\Delta_Q^{d+} & =\{(q,v)\in Q\times Q\mid\omega_{d+}^\mu(q,v)=0,~1\leq\mu\leq m\},\\
\Delta_Q^{d-} & =\{(v,q)\in Q\times Q\mid\omega_{d-}^\mu(v,q)=0,~1\leq\mu\leq m\}.
\end{align*}

\subsection{Variational formulation of discrete Dirac mechanics}\label{sec:lagrangedirac}

Let $Q$ be a vector space, $L_d:Q\times Q\times(0,1)\to\mathbb R$ be a discrete Lagrangian and $\Delta_Q\subset TQ$ be a (possibly nonholonomic) constraint. Let $N\in\mathbb Z^+$ be the number of steps. As usual, there are two possible choices for the discrete action functional arising from the two different types of generating function that can be used to discretize (cf. \cite[Definition 8.2]{LeOh2011}:
\begin{enumerate}[(i)]
    \item A sequence $\{(q_k,v_k,p_k)\in Q\times Q\times Q^*\mid 0\leq k\leq N\}$ is \emph{stationary} or \emph{critical} for the \emph{$(+)$-discrete Lagrange--d'Alembert--Pontryagin action} if
    \begin{equation*}
    \delta\sum_{k=0}^{N-1}L_d(q_k,v_k,h)+p_{k+1}\cdot(q_{k+1}-v_k)=0,
    \end{equation*}
    together with the constraint $(q_k,q_{k+1})\in\Delta_Q^{d+}$, $0\leq k\leq N-1$, for free variations such that $\delta q_0=\delta q_N=0$ and $\delta q_k\in\Delta_Q(q_k)$, $1\leq k\leq N-1$.
    \item A sequence $\{(v_k,q_k,p_k)\in Q\times Q\times Q^*\mid 0\leq k\leq N\}$ is \emph{stationary} or \emph{critical} for the \emph{$(-)$-discrete Lagrange--d'Alembert--Pontryagin action principle} if
    \begin{equation*}
    \delta\sum_{k=0}^{N-1}L_d(v_{k+1},q_{k+1},h)-p_k\cdot(q_k-v_{k+1})=0,
    \end{equation*}
    together with the constraint $(q_{k+1},q_k)\in\Delta_Q^{d-}$, $0\leq k\leq N-1$, for free variations such that $\delta q_0=\delta q_N=0$ and $\delta q_k\in\Delta_Q(q_k)$, $1\leq k\leq N-1$.
\end{enumerate}

%In order to write the discrete dynamical equations, recall that the partial derivatives of the discrete Lagrangian,
%\begin{align*}
%& D_1L_d:Q\times Q\times(0,1)\to Q^*,\\
%& D_2L_d:Q\times Q\times(0,1)\to Q^*,\\
%& D_3L_d:Q\times Q\times(0,1)\to\mathbb R,
%\end{align*}
%are defined as
%\begin{align*}
%D_1 L_d(q,v,h)\cdot\delta q & =\left.\frac{d}{d\epsilon}\right|_{\epsilon=0}L_d(q+\epsilon\,\delta q,v,h),\quad \delta q\in Q,\\
%D_2 L_d(q,v,h)\cdot\delta v & =\left.\frac{d}{d\epsilon}\right|_{\epsilon=0}L_d(q,v+\epsilon\,\delta v,h),\quad \delta v\in Q,\\
%D_3 L_d(q,v,h)\cdot\delta h & =\left.\frac{d}{d\epsilon}\right|_{\epsilon=0}L_d(q,v,h+\epsilon\,\delta h),\quad \delta h\in\mathbb R,
%\end{align*}
%for each $(q,v,h)\in Q\times Q\times(0,1)$.

\begin{theorem}
\begin{enumerate}
    \item A sequence $\{(q_k,v_k,p_k)\in Q\times Q\times Q^*\mid 0\leq k\leq N\}$ is stationary for the $(+)$-discrete Lagrange--d'Alembert--Pontryagin action if it satisfies the \emph{$(+)$-discrete nonholonomic implicit Euler--Lagrange equations} (also called \emph{$(+)$-discrete Lagrange--Dirac equations}:
    \begin{equation}\label{eq:+discreteLdAP}
    \left\{\begin{array}{ll}
    p_{k+1}=D_2 L_d(q_k,v_k,h),\qquad & \\
    q_{k+1}=v_k,& \\
    p_k+D_1 L_d(q_k,v_k,h)\in\Delta_Q^\circ(q_k), & k\neq 0,\\
    (q_k,v_k)\in\Delta_Q^{d+}, & k\neq0,
    \end{array}\right.\qquad 0\leq k\leq N-1.
    \end{equation}
    \item A sequence $\{(v_k,q_k,p_k)\in Q\times Q\times Q^*\mid 0\leq k\leq N\}$ is stationary for the $(-)$-discrete Lagrange--d'Alembert--Pontryagin action if it satisfies the \emph{$(-)$-discrete nonholonomic implicit Euler--Lagrange equations} (also called \emph{$(-)$-discrete Lagrange--d'Alembert--Pontryagin equations}):
    \begin{equation}\label{eq:-discreteLdAP}
    \left\{\begin{array}{ll}
    p_k=-D_1 L_d(v_{k+1},q_{k+1},h),\qquad & \\
    q_k=v_{k+1},& \\
    p_{k+1}-D_2 L_d(v_{k+1},q_{k+1},h)\in\Delta_Q^\circ(q_{k+1}), & k\neq 0,\\
    (v_{k+1},q_{k+1})\in\Delta_Q^{d-}, & k\neq0,
    \end{array}\right.\qquad0\leq k\leq N-1.
    \end{equation}
\end{enumerate}

\end{theorem}

\section{Discrete nonholonomic implicit Lagrangian systems with collisions}

In the following, $Q$ will be assumed to be a vector space. Therefore, the tangent and cotangent bundles are given by $TQ=Q\times Q$ and $T^*Q=Q\times Q^*$. The \emph{admissible set} is some submanifold $S\subset Q$ such that $\dim Q =\dim S$. For each point on the boundary of the admissible set, $q\in\partial S$, there exists a vector subspace $V_q\subset Q$ such that $T_q\partial S=\{q\}\times V_q$ and $T_q^*\partial S=\{q\}\times V_q^*$. Hence, the tangent and cotangent bundles of $\partial S$ are given by $T\partial S=\bigcup_{q\in\partial S}\{q\}\times V_q$ and $T^*\partial S=\bigcup_{q\in\partial S}\{q\}\times V_q^*$. The natural inclusion $\imath_{S}:\partial S\to S$ induces canonical morphisms
\begin{align*}
T\imath_{S} :T\partial S\to TQ|_{\partial S}=\partial S\times Q,\qquad\imath_{S}^*:T^*Q|_{\partial S}=\partial S\times Q^*\to T^*\partial S.
\end{align*}
A choice of a linear projection $(\pi_S)_q:Q\to V_q$ for each $q\in\partial S$ determines a vector bundle projection
\begin{align*}%\label{eq:piS}
\pi_S:TQ|_{\partial S}=\partial S\times Q\to T\partial S,\quad v_q=(q,v)\mapsto\pi_S(v_q)=(q,(\pi_S)_q(v)).
\end{align*}
Note that $\pi_S$ is a left inverse of $T\imath_{S} $, i.e., $\pi_S\circ T\imath_{S} ={\rm id}_{T\partial S}$; and, thus, the adjoint map $\pi_S^*:T^*\partial S\to T^*Q|_{\partial S}=\partial S\times Q^*$ is a right inverse of $\imath_{S}^*$, i.e., $\imath_{S}^*\circ\pi_S^*={\rm id}_{T^*\partial S}$. The situation is summarized in the following diagrams:
\begin{align*}%\label{eq:imathSpiS}
\begin{tikzpicture}
\matrix (m) [matrix of math nodes,row sep=6em,column sep=4em,minimum width=2em, ampersand replacement=\&]
{ T\partial S \& TQ|_{\partial S}=\partial S\times Q,\\};
\path[-stealth]
(m-1-1) edge [] node [above] {$T\imath_S$} (m-1-2)
(m-1-2) edge [bend left] node [below] {$\pi_S$} (m-1-1);
\end{tikzpicture}\quad
\begin{tikzpicture}
\matrix (m) [matrix of math nodes,row sep=6em,column sep=4em,minimum width=2em, ampersand replacement=\&]
{T^*Q|_{\partial S}=\partial S\times Q^* \& T^*\partial S.\\};
\path[-stealth]
(m-1-1) edge [] node [above] {$\imath_S^*$} (m-1-2)
(m-1-2) edge [bend left] node [below] {$\pi_S^*$} (m-1-1);
\end{tikzpicture}
\end{align*}

Let $L_d:Q\times Q\times(0,1)\to\mathbb R$ be a discrete (possibly degenerate) Lagrangian, where the third coordinate represents the timestep, which will be denoted by $h$.

\subsection{Discrete configuration and phase spaces}

Let $[t,T]\subset\mathbb R$ be an interval and $N\in\mathbb Z^+$ be the number of steps. The discrete interval is $\{t_k=t+kh\mid 0\leq k\leq N\}$, where $h=(T-t)/N$ is the timestep. Note that $t_0=t$ and $t_N=T$. We only consider one collision at a fixed step $i\in\{0,\dots,N\}$ (which is assumed to be known) occurring at a fixed time $\tilde t=t_i+\alpha h$ for some $\alpha\in(0,1)$ (which is assumed to be unknown). Given $\tilde\alpha\in(0,1)$, the \emph{discrete interval with a unique collision} is defined as
\begin{equation*}
I(i,\tilde\alpha)=\{t_0,\dots,t_i,t_i+\tilde\alpha h,t_{i+1},\hdots,t_N\}.
\end{equation*}
In the same vein, the \emph{discrete path space} is $\Omega_d(S,i,\tilde\alpha)=(0,1)\times\mathcal S_d(i,\tilde\alpha)$, where
\begin{align}\label{eq:Sd}
\mathcal S_d(i,\tilde\alpha) & =\big\{q_d:I(i,\tilde\alpha)\to S\mid q_d(t_k)\in{\rm int}\,S,~0\leq k\leq N,~q_d(t_i+\tilde\alpha h)\in\partial S\big\},
\end{align}
with ${\rm int}\,S=S-\partial S$ denoting the interior of $S$. Note that $\Omega_d(S,i,\tilde\alpha)\simeq(0,1)\times({\rm int}\,S)^{N+1}\times\partial S$. For brevity, we denote $q_k=q_d(t_k)$ for each $0\leq k\leq N$ and $\tilde q=q_d(t_i+\tilde\alpha h)$. Analogously, let us introduce a  discrete path space where the admissible set is disregarded, $\mathcal Q_d(i,\tilde\alpha)=\left\{v_d:I(i,\tilde\alpha)\to Q\right\}$. Again, for brevity we denote $v_k=v_d(t_k)$, $0\leq k\leq N$, and $\tilde v=v_d(t_i+\tilde\alpha h)$. Note that
\begin{align*}
\mathcal S_d(i,\tilde\alpha)\times\mathcal Q_d(i,\tilde\alpha)& =\{(q_d,v_d):I(i,\tilde\alpha)\to S\times Q\mid q_k\in{\rm int}\,S,~1\leq k\leq N,~\tilde q\in\partial S\}    
\end{align*}
defines a vector bundle over $\mathcal S_d(i,\tilde\alpha)$. For the $(-)$-case, we will use $\mathcal Q_d(i,\tilde\alpha)\times\mathcal S_d(i,\tilde\alpha)$ instead, and denote its elements by $(v_d,q_d)$.

Given that $TS=S\times Q$ and $T^*S=S\times Q^*$, the tangent and cotangent spaces of $\mathcal S_d(i,\tilde\alpha)$ at $q_d\in\mathcal S_d(i,\tilde\alpha)$ are given by
\begin{align*}%\label{eq:TqdSd}
T_{q_d}\mathcal S_d(i,\tilde\alpha) %& =\left\{\nu_{q_d}:I(i,\tilde\alpha)\to TS\mid\pi_{TQ}\circ\nu_{q_d}=q_d,~ \nu_{q_d}(t_i+\tilde\alpha h)\in T_{\tilde q}\partial S\right\}\\
& =\left\{\nu_d:I(i,\tilde\alpha)\to Q\mid\nu_d(t_i+\tilde\alpha h)\in V_{\tilde q}\right\},\\%\label{eq:Tqd*Sd}
T_{q_d}^\star\mathcal S_d(i,\tilde\alpha) %& =\left\{ p_{q_d}:I(i,\tilde\alpha)\to T^*S\mid\pi_{T^*Q}\circ p_{q_d}=q_d,~ p_{q_d}(t_i+\tilde\alpha h)\in T_{\tilde q}^*\partial S\right\}\\
& =\left\{p_d:I(i,\tilde\alpha)\to Q^*\mid  p_d(t_i+\tilde\alpha h)\in V_{\tilde q}^*\right\}.
\end{align*}
As above, we denote $\nu_k=\nu_d(t_k)$ and $p_k=p_d(t_k)$ for each $0\leq k\leq N$, as well as $\tilde\nu=\nu_d(t_i+\tilde\alpha h)$ and $\tilde p=p_d(t_i+\tilde\alpha h)$. As usual, the tangent and cotangent bundles are denoted by $T\mathcal S_d(i,\tilde\alpha)=\bigsqcup_{q_d\in\mathcal S(i,\tilde\alpha)}T_{q_d}\mathcal S_d(i,\tilde\alpha)$ and $T^\star\mathcal S_d(i,\tilde\alpha)=\bigsqcup_{q_d\in\mathcal S(i,\tilde\alpha)}T_{q_d}^\star\mathcal S_d(i,\tilde\alpha)$, respectively. %Note that, contrary to the continuous theory, $\mathcal S_d(i,\tilde\alpha)$ is a finite dimensional manifold and, thus, $T^\star\mathcal S_d(i,\tilde\alpha)$ is the actual topological dual bundle of $T\mathcal S_d(i,\tilde\alpha)$, with the dual pairing being
%\begin{equation*}
%\langle(q_d,p_d),(q_d,\nu_d)\rangle=\sum_{k=0}^N p_k\cdot\nu_k+\tilde p\cdot\tilde\nu,
%\end{equation*}
%for each $(q_d,\nu_d,p_d)\in T\mathcal S(i,\tilde\alpha)\oplus T^\star\mathcal S(i,\tilde\alpha)$.
%Lastly, the \emph{discrete Pontryagin bundle} is the Whitney sum of the vector bundles $\mathcal S_d(i,\tilde\alpha)\times\mathcal Q_d(i,\tilde\alpha)$ and $T^\star\mathcal S_d(i,\tilde\alpha)$ over $\mathcal S_d(i,\tilde\alpha)$, i.e.,
%\begin{align*}
%\mathcal P^d(i,\tilde\alpha,S) & =(\mathcal S_d(i,\tilde\alpha)\times\mathcal Q_d(i,\tilde\alpha))\oplus T^\star\mathcal S_d(i,\tilde\alpha)\\
%& =\big\{(q_d,v_d, p_d):I(i,\tilde\alpha)\to(S\times Q)\oplus T^*S\mid q_d\in\mathcal S_d(i,\tilde\alpha),~\tilde p\in V_{\tilde q}^*\big\}.    
%\end{align*}

In order to work with variations of a discrete path, we need to introduce the iterated bundles. Note that $T(TS)=S\times Q\times Q\times Q$ and $T(T^*S)=S\times Q^*\times Q\times Q^*$, whence, for each $(q_d,v_d,p_d)\in(\mathcal S_d(i,\tilde\alpha)\times\mathcal Q_d(i,\tilde\alpha))\oplus T^\star\mathcal S_d(i,\tilde\alpha)$, we have
\begin{align*}%\label{eq:TSdQd}
& T_{(q_d,v_d)}(\mathcal S_d(i,\tilde\alpha)\times\mathcal Q_d(i,\tilde\alpha))=\big\{(\delta q_d,\delta v_d):I(i,\tilde\alpha)\to Q\times Q\mid\delta\tilde q,\delta\tilde v\in V_{\tilde q}\big\},\\%\label{eq:TQ_dSd}
& T_{(v_d,q_d)}(\mathcal Q_d(i,\tilde\alpha)\times\mathcal S_d(i,\tilde\alpha))=\big\{(\delta v_d,\delta q_d):I(i,\tilde\alpha)\to Q\times Q\mid\delta\tilde v,\delta\tilde q\in V_{\tilde q}\big\},\\%\label{eq:TT*Sd}
& T_{(q_d,p_d)}(T^\star\mathcal S_d(i,\tilde\alpha))=\big\{(\delta q_d,\delta p_d):I(i,\tilde\alpha)\to Q\times Q^*\mid\delta\tilde q\in V_{\tilde q},~\delta\tilde p\in V_{\tilde q}^*\big\}.
\end{align*}
%As a result, $T_{(q_d,v_d,p_d)}\mathcal P^d(i,\tilde\alpha,S)=T_{(q_d,v_d)}(\mathcal S_d(i,\tilde\alpha)\times\mathcal Q_d(i,\tilde\alpha))\times T_{(q_d,p_d)}(T^\star\mathcal S_d(i,\tilde\alpha))$.
%where we have denoted $\delta\tilde q=\delta q_d(t_i+\tilde\alpha h)$, $\delta\tilde v=\delta v_d(t_i+\tilde\alpha h)$ and $\delta\tilde p=\delta p_d(t_i+\tilde\alpha h)$. Similarly, we denote $\delta q_k=\delta q_d(t_k)$, $\delta v_k=\delta v_d(t_k)$ and $\delta p_k=\delta p_d(t_k)$ for each $0\leq k\leq N$.

Given a (possibly nonholonomic) constrain distribution $\Delta_Q\subset TQ$, the constrained path spaces are defined as (recall \S\ref{sec:discreteconstraint})
\begin{align*}
\Delta_S^{d+}(i,\tilde\alpha) & =\{(q_d,v_d)\in\mathcal S_d(i,\tilde\alpha)\times\mathcal Q_d(i,\tilde\alpha)\mid(q_k,v_k),(\tilde q,\tilde v)\in\Delta_Q^{d+},~0\leq k\leq N\},\\
\Delta_S^{d-}(i,\tilde\alpha) & =\{(v_d,q_d)\in\mathcal Q_d(i,\tilde\alpha)\times\mathcal S_d(i,\tilde\alpha)\mid(v_k,q_k),(\tilde v,\tilde q)\in\Delta_Q^{d-},~0\leq k\leq N\}.
\end{align*}
Analogously, for each $(q_d,v_d)\in\mathcal S_d(i,\tilde\alpha)\times\mathcal Q_d(i,\tilde\alpha)$, we introduce the following path spaces:
\begin{align*}
& \Delta_{TS}^{d+}(i,\tilde\alpha)(q_d,v_d)=\{(\delta q_d,\delta v_d)\in T_{(q_d,v_d)}(\mathcal S_d(i,\tilde\alpha)\times\mathcal Q_d(i,\tilde\alpha))\mid\delta q_k\in\Delta_Q(q_k),~0\leq k\leq N,~\delta\tilde q\in\Delta_Q(\tilde q)\},\\
& \Delta_{TS}^{d-}(i,\tilde\alpha)(v_d,q_d)=\{(\delta v_d,\delta q_d)\in T_{(v_d,q_d)}(\mathcal Q_d(i,\tilde\alpha)\times\mathcal S_d(i,\tilde\alpha))\mid\delta q_k\in\Delta_Q(q_k),~0\leq k\leq N,~\delta\tilde q\in\Delta_Q(\tilde q)\}.
\end{align*}

\color{black}

\subsection{Discrete nonholonomic implicit Euler--Lagrange equations with collisions}\label{sec:EL_eqs}

The two possible choices for the discrete action functional are given by:
\begin{enumerate}[(i)]
    \item The \emph{$(+)$-discrete Lagrange--Pontryagin action}, $\mathbb S_{d+}:(0,1)\times(\mathcal S_d(i,\tilde\alpha)\times\mathcal Q_d(i,\tilde\alpha))\oplus T^\star\mathcal S_d(i,\tilde\alpha)\to\mathbb R$, is defined as
    \begin{align}\label{eq:Sd+}
    \mathbb S_{d+}(\alpha,q_d,v_d,p_d)= & \sum_{\substack{k=0\\k\neq i}}^{N-1}\big(L_d(q_k,v_k,h)+p_{k+1}\cdot(q_{k+1}-v_k)\big)+L_d(q_i,v_i,\alpha h)\\\nonumber
    & +(\pi_S)_{\tilde q}^*(\tilde p)\cdot(\tilde q-v_i)+L_d(\tilde q,\tilde v,(1-\alpha)h)+p_{i+1}\cdot(q_{i+1}-\tilde v).
    \end{align}
    \item The \emph{$(-)$-discrete Lagrange--Pontryagin action}, $\mathbb S_{d-}:(0,1)\times(\mathcal Q_d(i,\tilde\alpha)\times\mathcal S_d(i,\tilde\alpha))\oplus T^\star\mathcal S_d(i,\tilde\alpha)\to\mathbb R$, is defined as
    \begin{align}\label{eq:Sd-}
    \mathbb S_{d-}(\alpha,v_d,q_d,p_d)= & \sum_{\substack{k=0\\k\neq i}}^{N-1}\big(L_d(v_{k+1},q_{k+1},h)-p_k\cdot(q_k-v_{k+1})\big)+L_d(\tilde v,\tilde q,\alpha h)\\\nonumber
    & -p_i\cdot(q_i-\tilde v)+L_d(v_{i+1},q_{i+1},(1-\alpha)h)-(\pi_S)_{\tilde q}^*(\tilde p)\cdot(\tilde q-v_{i+1}).
    \end{align}
\end{enumerate}

The following definition introduces the \emph{discrete Hamilton--d'Alembert--Pontryagin principle}.

\begin{definition}\label{def:variationalprinciple}
\begin{enumerate}[(i)]
    \item A discrete path
    \begin{equation*}
    \texttt c_{d+}=(\alpha,q_d,v_d,p_d)\in(0,1)\times\Delta_S^{d+}(i,\tilde\alpha)\oplus T^*\mathcal S_d(i,\tilde\alpha)
    \end{equation*}
    is \emph{stationary} (or \emph{critical}) for the $(+)$-discrete Lagrange--d'Alembert--Pontryagin action if it satisfies
    \begin{equation*}
    {\rm d}\,\mathbb S_{d+}(\texttt c_{d+})(\delta\texttt c_{d+})=0,
    \end{equation*}
    for each variation $\delta\texttt c_{d+}=(\delta\alpha,\delta q_d,\delta v_d,\delta p_d)\in T_\alpha(0,1)\times\Delta_{TS}^{d+}(i,\tilde\alpha)(q_d,v_d)\times T_{(q_d,p_d)}(T^\star\mathcal S_d(i,\tilde\alpha))$ vanishing at the endpoints $\delta q_0=\delta q_N=0$.
    \item A discrete path
    \begin{equation*}
    \texttt c_{d-}=(\alpha,v_d,q_d,p_d)\in(0,1)\times\Delta_S^{d-}(i,\tilde\alpha)\oplus T^*\mathcal S_d(i,\tilde\alpha)
    \end{equation*}
    is \emph{stationary} or \emph{critical} for the $(-)$-discrete Lagrange--d'Alembert--Pontryagin action if it satisfies
    \begin{equation*}
    {\rm d}\,\mathbb S_{d-}(\texttt c_{d-})(\delta\texttt c_{d-})=0,
    \end{equation*}
    for each variation $\delta\texttt c_{d-}=(\delta\alpha,\delta v_d,\delta q_d,\delta p_d)\in T_\alpha(0,1)\times\Delta_{TS}^{d-}(i,\tilde\alpha)(v_d,q_d)\times T_{(q_d,p_d)}(T^\star\mathcal S_d(i,\tilde\alpha))$ vanishing at the endpoints $\delta q_0=\delta q_N=0$.
\end{enumerate} 
\end{definition}

The condition $(\delta q_d,\delta v_d)\in\Delta_{TS}^{d+}(i,\tilde\alpha)(q_d,v_d)$ implies that $\delta q_k\in\Delta_Q(q_k)$, $0\leq k\leq N$, and $\delta\tilde q\in\Delta_Q(\tilde q)$, and analogous for $(\delta v_d,\delta q_d)\in\Delta_{TS}^{d-}(i,\tilde\alpha)(v_d,q_d)$. These conditions are imposed \emph{after} taking variations inside the summation. We are ready to compute the discrete nonholonomic implicit Euler--Lagrange equations with collisions.

\begin{theorem}\label{theorem:discrete_ELeqs}
Consider the following discrete paths:
\begin{align*}
\texttt c_{d+}=(\alpha,q_d,v_d,p_d)\in(0,1)\times\Delta_S^{d+}(i,\tilde\alpha)\oplus T^*\mathcal S_d(i,\tilde\alpha),\\
\texttt c_{d-}=(\alpha,v_d,q_d,p_d)\in(0,1)\times\Delta_S^{d-}(i,\tilde\alpha)\oplus T^*\mathcal S_d(i,\tilde\alpha).
\end{align*}
Then, the following statements hold:
\begin{enumerate}[(i)]
    \item $\texttt c_{d+}$ is stationary for $\mathbb S_{d+}$ if and only if it satisfies the $(+)$-discrete nonholonomic implicit Euler--Lagrange equations \eqref{eq:+discreteLdAP} for $k\neq i$, together with the $(+)$-discrete conditions for the \emph{elastic impact}:
    \begin{align*}
    \left\{\begin{array}{l}
    D_1 L_d(q_i,v_i,\alpha h)+p_i\in\Delta_Q^\circ(q_i),\\
    (q_i,v_i)\in\Delta_Q^{d+},\\
    \tilde q=v_i,\\
    \tilde q\in\partial S,\\
    D_3 L_d(q_i,v_i,\alpha h)=D_3 L_d(\tilde q,\tilde v,(1-\alpha)h),\\
    (\imath_{S})_{\tilde q}^*(D_1 L_d(\tilde q,\tilde v,(1-\alpha)h)+\tilde p)\in(\imath_S)_{\tilde q}^*\left(\Delta_Q^\circ(\tilde q)\right),\\
    (\tilde q,\tilde v)\in\Delta_Q^{d+},\\
    (\pi_S)_{\tilde q}^*(\tilde p)=D_2 L_d(q_i,v_i,\alpha h),\\
    p_{i+1}=D_2 L_d(\tilde q,\tilde v,(1-\alpha)h),\\
    q_{i+1}=\tilde v.
    \end{array}\right.
    \end{align*}
    \item $\texttt c_{d-}$ is stationary for $\mathbb S_{d-}$ if and only if it satisfies the $(-)$-discrete nonholonomic implicit Euler--Lagrange equations \eqref{eq:-discreteLdAP} for $k\neq i$, together with the $(-)$-discrete conditions for the \emph{elastic impact}:
    \begin{align*}
    \left\{\begin{array}{l}
    (\imath_S)_{\tilde q}^*(-D_2 L_d(\tilde v,\tilde q,\alpha h)+\tilde p)\in(\imath_S)_{\tilde q}^*\left(\Delta_Q^\circ(\tilde q)\right),\\
    (\tilde v,\tilde q)\in\Delta_Q^{d-},\\
    q_i=\tilde v,\\
    \tilde q\in\partial S,\\
    D_3 L_d(\tilde v,\tilde q,\alpha h)=D_3 L_d(v_{i+1},q_{i+1},(1-\alpha)h),\\
    -D_2 L_d(v_{i+1},q_{i+1},(1-\alpha)h)+p_{i+1}\in\Delta_Q^\circ(q_{i+1}),\\
    (v_{i+1},q_{i+1})\in\Delta_Q^{d-},\\
    p_i=-D_1 L_d(\tilde v,\tilde q,\alpha h),\\
    (\pi_S)_{\tilde q}^*(\tilde p)=-D_1 L_d(q_{i+1},v_{i+1},(1-\alpha)h),\\
    \tilde q=v_{i+1}.
    \end{array}\right.
    \end{align*}
\end{enumerate}
\end{theorem}

\begin{proof}
By computing the variation of the $(+)$-discrete action functional and rearranging terms, we obtain:
\begin{align*}
\delta\mathbb S_{d+}(\texttt c_{d+})(\delta\texttt c_{d+})= & \sum_{\substack{k=0\\k\neq i-1}}^{N-2}(D_1 L_d(q_{k+1},v_{k+1},h)+p_{k+1})\cdot\delta q_{k+1}+\sum_{\substack{k=0\\k\neq i}}^{N-1}(D_2 L_d(q_k,v_k,h)-p_{k+1})\cdot\delta v_k\\
& +(D_1 L_d(q_i,v_i,\alpha h)+p_i)\cdot\delta q_i+(D_2 L_d(q_i,v_i,\alpha h)-(\pi_S)_{\tilde q}^*(\tilde p))\cdot\delta v_i\\
& +(\imath_S)_{\tilde q}^*(D_1 L_d(\tilde q,\tilde v,(1-\alpha)h)+\tilde p)\cdot\delta\tilde q+(D_2 L_d(\tilde q,\tilde v,(1-\alpha)h)-p_{i+1})\cdot\delta\tilde v\\
& +\sum_{\substack{k=0\\k\neq i}}^{N-1}\delta p_{k+1}\cdot(q_{k+1}-v_k)+(\pi_S)_{\tilde q}^*(\delta\tilde p)\cdot(\tilde q-v_i)+\delta p_{i+1}\cdot(q_{i+1}-\tilde v)\\
& +h\,(D_3 L_d(q_i,v_i,\alpha h)-D_3 L_d(\tilde q,\tilde v,(1-\alpha)h))\,\delta\alpha.
\end{align*}
We conclude by recalling that $\delta\alpha\in T_\alpha(0,1)\simeq\mathbb R$, $\delta q_k\in\Delta_Q(q_k)$, $1\leq k\leq N-1$, $\delta v_k,\delta\tilde v\in Q$, $\delta p_k\in Q^*$, $0\leq k\leq N$, $\delta\tilde q\in\Delta_Q(\tilde q)\cap V_{\tilde q}$ and $\delta\tilde p\in V_{\tilde q}^*$. The conditions $\tilde q\in\partial S$, $(q_k,v_k)\in\Delta_Q^{d+}$ and $(\tilde q,\tilde v)\in\Delta_Q^{d+}$ are by construction. To conclude, note that $V_{\tilde q}^\circ\cap V_{\tilde q}^*=\{0\}$, whence
\begin{equation*}
\left(\Delta_Q(\tilde q)\cap V_{\tilde q}\right)^\circ\cap V_{\tilde q}^*=\left(\Delta_Q^\circ(\tilde q)+V_{\tilde q}^\circ\right)\cap V_{\tilde q}^*=\Delta_Q^\circ(\tilde q)\cap V_{\tilde q}^*.
\end{equation*}
%\begin{equation*}
%(\imath_S)_{\tilde q}^*(D_1 L_d(\tilde q,\tilde v,(1-\alpha)h))+\tilde p\in.
%\end{equation*}
The $(-)$-discrete case follows from an analogous computation.
\end{proof}

As usual, the $(+)$-discrete equations are utilized to perform forward in time simulations, whereas the $(-)$-discrete equations are appropriate to simulate backwards in time. By utilizing Lagrange multipliers and the definition of $\Delta_Q^{+}$, this theorem may be implemented as described in Algorithm \ref{alg}. The backwards in time implementation is analogous. The algorithm checks at each iteration the admissibility of the solution. Whenever the new state is not admissible, it is removed and the conditions for the elastic impact are solved.

\begin{algorithm}[!ht]
\caption{Forward implementation of Theorem \ref{theorem:discrete_ELeqs}.}\label{alg}
\begin{enumerate}[$\bullet$]
    \item\textbf{input}: $(q_0,v_0)\in S\times Q$, $p_0=D_2 L_d(q_0,v_0,h)$.
    \item \textbf{while} $t_k<t_N$:
    \begin{enumerate}[$\circ$]
        \item solve:
        \begin{equation*}
        \left\{\begin{array}{l}
        p_{k+1}=D_2 L_d(q_k,v_k,h),\\
        q_{k+1}=v_k,\\
        D_1 L_d(q_{k+1},v_{k+1},h)+p_{k+1}=\lambda_{k+1,\mu}\,\omega_{q_{k+1}}^\mu,\\
        \omega_{d+}^\mu(q_{k+1},v_{k+1})=0,~1\leq\mu\leq m,
        \end{array}\right.
        \end{equation*}
        for $(q_{k+1},v_{k+1},p_{k+1})\in Q\times Q\times Q^*$ and $\{\lambda_{k+1,\mu}\in\mathbb R\mid1\leq\mu\leq m\}$.
        \item \textbf{if} $q_{k+1}\not\in S$:
    \begin{enumerate}[$\diamond$]
        \item delete: $v_k$, $(q_{k+1},v_{k+1},p_{k+1})$.
        \item solve:
        \begin{equation*}
        \left\{\begin{array}{l}
        D_1 L_d(q_k,v_k,\alpha h)+p_k=\lambda_{k,\mu}\,\omega_{q_k}^\mu,\\
        \omega_{d+}^\mu(q_k,v_k)=0,~1\leq\mu\leq m,\\
        \tilde q=v_k,\\
        \tilde q\in\partial S,
        \end{array}\right.
        \end{equation*} 
        for $(\alpha,\tilde q,v_k)\in(0,1)\times\partial S\times Q$ and $\{\lambda_{k,\mu}\in\mathbb R\mid 1\leq\mu\leq m\}$.
        \item solve: 
        \begin{equation*}
        \left\{\begin{array}{l}
        D_3 L_d(q_k,v_k,\alpha h)=D_3 L_d(\tilde q,\tilde v,(1-\alpha)h),\\
        (\imath_{S})_{\tilde q}^*\left(D_1 L_d(\tilde q,\tilde v,(1-\alpha)h)+\tilde p-\tilde\lambda_\mu\,\omega_{\tilde q}^\mu\right)=0,\\
        \omega_{d+}^\mu(\tilde q,\tilde v)=0,~1\leq\mu\leq m,\\
        (\pi_S)_{\tilde q}^*(\tilde p)=D_2 L_d(q_k,v_k,\alpha h),
        \end{array}\right.
        \end{equation*}
        for $(\tilde v,\tilde p)\in V_{\tilde q}\times V_{\tilde q}^*$ and $\left\{\tilde\lambda_\mu\in\mathbb R\mid1\leq\mu\leq m\right\}$.
        \item solve:
        \begin{equation*}
        \left\{\begin{array}{l}
        p_{k+1}=D_2 L_d(\tilde q,\tilde v,(1-\alpha)h),\\
        q_{k+1}=\tilde v.
        \end{array}\right.
        \end{equation*}
        for $(q_{k+1},p_{k+1})\in S\times Q^*$.
        \item solve:
        \begin{equation*}
        \left\{\begin{array}{l}
        D_1L_d(q_{k+1},v_{k+1},h)+p_{k+1}=\lambda_{k+1,\mu}\,\omega_{q_{k+1}}^\mu,\\
        \omega_{d+}^\mu(q_{k+1},v_{k+1})=0,~1\leq\mu\leq m.
        \end{array}\right.
        \end{equation*}
        for $v_{k+1}\in Q$ and $\{\lambda_{k+1,\mu}\in\mathbb R\mid1\leq\mu\leq m\}$.
    \end{enumerate}
    \end{enumerate}
\end{enumerate}
\end{algorithm}

\section{Examples}

\subsection{Bouncing particle}

Consider a particle in two dimensions subject to the gravitational force that collides with the floor. The configuration space and the admissible set are $Q=\mathbb R^2$ and $S=\{q=(x,y)\in\mathbb R^2\mid y\geq 0\}$, respectively, whose boundary is given by $\partial S=\{\tilde q=(x,y)\in\mathbb R^2\mid y=0\}$. Henceforth, the vectors and covectors are expressed in the following bases: $TQ={\rm span}\{\partial_x,~\partial_y\}$, $T^*Q={\rm span}\{dx,~dy\}$, $T\partial S={\rm span}\{\partial_x\}$ and $T^*\partial S={\rm span}\{dx\}$. The system is unconstrained, i.e., $\Delta_Q=Q$ and, thus, $\Delta_Q^{d+}=Q\times Q$. By denoting $\tilde q=(x,0)\in\partial S$, we have the following maps:
\begin{align*}
T_{\tilde q}\imath_S:T\partial S\to T\mathbb R^2 |_{\partial S},\quad & \tilde v=(\tilde v_x)\mapsto T_{\tilde q}\imath_S(\tilde v)=(\tilde v_x,0),\\
(\imath_S^*)_{\tilde q}:T^*\mathbb R^2 |_{\partial S}\to T^*\partial S,\quad & p=(p_x,p_y)\mapsto(\imath_S^*)_{\tilde q}(p)=(p_x),\\
(\pi_S)_{\tilde q}:T\mathbb R^2 |_{\partial S}\to T\partial S,\quad & v=(v_x,v_y)\mapsto(\pi_S)_{\tilde q}(v)=(v_x),\\
(\pi_S^*)_{\tilde q}:T^*\partial S\to T^*\mathbb R^2 |_{\partial S},\quad & \tilde p=(\tilde p_x)\mapsto(\pi_S^*)_{\tilde q}(\tilde p)=(\tilde p_x,0).
\end{align*}
The Lagrangian of the system $L:T\mathbb R^2 \to\mathbb R$ consists of the kinetic energy together with the gravitational potential, i.e.,
\begin{equation*}
L(q,v)=\frac{1}{2}m(v_x^2+v_y^2)-mgy,
\end{equation*}
for each $(q,v)=((x,y),(v_x,v_y))\in T\mathbb R^2$, where $m>0$ and $g>0$ are the mass and the gravitational acceleration, respectively. Lastly, we consider the discrete Lagrangian $L_d:\mathbb R^2 \times\mathbb R^2 \times(0,1)\to\mathbb R$ defined as
\begin{align*}
& L_d(q,v,h)=h\,L\left(\frac{q+v}{2},\frac{v-q}{h}\right)=\frac{1}{2h}m\left((v_x-x)^2+(v_y-y)^2\right)-\frac{1}{2}hmg(y+v_y).
\end{align*}
Therefore, the relation between the continuous and the discrete initial conditions is given by
\begin{equation*}
q_0=q(0)-\frac{h}{2}v(0),\qquad v_0=q(0)+\frac{h}{2}v(0).
\end{equation*}

\subsection{Bouncing 2-dimensional rigid body}

Let us extend the previous example to a 2-dimensional rigid body subject to the gravitational force that collides with the floor and rotates about a perpendicular axis. The configuration space is the Euclidean group in two dimensions, i.e., $Q={\rm SE}(2)={\rm SO}(2)\ltimes\mathbb R^2$, where we pick coordinates $q=(\theta,x,y)\in Q$ with $\theta$ and $(x,y)$ denoting the angle of rotation and the position of the axis of rotation (which we assume to be inside the body). The system is unconstrained, i.e., $\Delta_Q=Q$ and, thus, $\Delta_Q^{d+}=Q\times Q$. The distance between the axis and the edge of the body is given by a piecewise smooth function $\phi:[0,2\pi]\to\mathbb R^+$ such that $\phi(0)=\phi(2\pi)$. For instance,
\begin{enumerate}
    \item $\phi(\theta)=l(|\sin\theta|+|\cos\theta|)$ corresponds to a four-point star-shaped rigid body of length $2l>0$ rotating about its center \cite[Eq. (81)]{FeMaOrWe2003}, and
    \item $\phi(\theta)=\sqrt{a^2\,\sin^2\theta+b^2\,\cos^2\theta}$ corresponds to an ellipse with principal semi-axes of lengths $a,b>0$ rotating about its center \cite[Proposition 2.2]{TrLe2024}.
\end{enumerate}
The admissible set is given by
\begin{equation*}
S=\{q=(\theta,x,y)\in{\rm SE}(2)\mid y\geq\phi(\theta)\}. %\partial S=\{q=(\theta,x,y)\in{\rm SE}(2)\mid y=\phi(\theta)\}
\end{equation*}
It can be checked that the following are bases for the tangent and cotangent bundles of the admissible set, respectively:
\begin{align*}
T\partial S & ={\rm span}\{e_1=\partial_x,~e_2=\partial_\theta+\dot\phi(\theta)\,\partial_y\},\\
T^*\partial S & ={\rm span}\{e^1=dx,~e^2=\psi(\theta)\,(d\theta+\dot\phi(\theta)\,dy)\},
\end{align*}
where $\psi(\theta)=1/(1+\dot\phi(\theta)^2)$. It is readily seen that
\begin{align*}
T\imath_S:T\partial S\to T{\rm SE}(2)|_{\partial S},\quad & (q,v)\mapsto T\imath_S(q,v)=v_2\,\partial_\theta+v_1\,\partial_x+\dot\phi(\theta)\,v_2\,\partial_y,%=v_1\,e_1+v_2\,e_2,
\\
\imath_S^*:T^*{\rm SE}(2)|_{\partial S}\to T^*\partial S,\quad & (q,p)\mapsto \imath_S^*(q,p) %=\frac{p_\theta-\dot\phi(\theta)\,p_y}{1+\dot\phi(\theta)^2}\,d\theta+p_x\,dx-\dot\phi(\theta)\,\frac{p_\theta-\dot\phi(\theta)\,p_y}{1+\dot\phi(\theta)^2}\,dy
=p_x\,e^1+(p_\theta+\dot\phi(\theta)\,p_y)\,e^2,
\end{align*}
for each $v=v_1\,e_1+v_2\,e_2\in T_q\partial S$ and $p=p_\theta\,d\theta+p_x\,dx+p_y\,dy\in T_q^*{\rm SE}(2)$. Similarly, the projection and its adjoint are chosen as
\begin{align*}
\pi_S:T{\rm SE}(2)|_{\partial S}\to T\partial S,\quad & (q,v)\mapsto\pi_S(q,v)%=v_\theta\,\partial_\theta+v_x\,\partial_x-\dot\phi(\theta)\,v_\theta\,\partial_y
=v_x\,e_1+v_\theta\,e_2,\\
\pi_S^*:T^*\partial S\to T^*{\rm SE}(2)|_{\partial S},\quad & (q,p)\mapsto\pi_S^*(q,p)=p_2\,d\theta+p_1\,dx,
\end{align*}
for each $v=v_\theta\,\partial_\theta+v_x\,\partial_x+v_y\,\partial_y\in T_q{\rm SE}(2)$ and $p=p_1\,e^1+p_2\,e^2\in T_q^*\partial S$.

The Lagrangian of the system, $L:T{\rm SE}(2)\to\mathbb R$, consists of the kinetic energy (translational and rotational) together with the gravitational potential, i.e.,
\begin{equation*}
L(q,v)=\frac{1}{2}m(v_x^2+v_y^2)+\frac{1}{2}I_\phi\,v_\theta^2-mgy,\qquad(q,v)\in T{\rm SE}(2),
\end{equation*}
where $m>0$ and $I_\phi>0$ are the mass and the moment of inertia with respect to the $z$ axis about the center of the body, respectively, of the object, and $g>0$ is the gravitational acceleration. The discrete Lagrangian $L_d:{\rm SE}(2)\times{\rm SE}(2)\times(0,1)\to\mathbb R$ is chosen as in the previous example:

\begin{align*}
L_d(q,v,h) & =h\,L\left(\frac{q+v}{2},\frac{v-q}{h}\right)=\frac{m}{2h}\left((v_x-x)^2+(v_y-y)^2\right)+\frac{I_\phi}{2h}(v_\theta-\theta)^2-hmg\frac{y+v_y}{2}.
\end{align*}

By applying Algorithm \ref{alg} to this discrete Lagrangian, the evolution of an ellipse is simulated. The numerical experiments are carried out with the following choice of parameters: $m=1$, $g=9.8$, $a=1$, $b=0.5$. The moment of inertia of an ellipse is $I_\phi=m(a^2+b^2)/4$. Lastly, we picked  $h=10^{-2}$ as time-step, as well as $q(0)=(\theta(0),x(0),y(0))=(\pi/2,0,3.5)$ and $v_q(0)=(v_\theta(0),v_x(0),v_y(0))=(-3,2,0)$ as initial conditions. In Figure \ref{fig:trajectory}, the trajectory of the ellipse is simulated for two seconds, including one impact with the floor. The dashed black line indicates the trajectory of the center and dashed red lines indicate the principal axes of the ellipse. In Figure \ref{fig:energy}, we plot the energy evolution of the system. As expected, there is some slight energy fluctuation on the long time, but the total energy remains nearly constant. An animated simulation of the system can be found \href{https://youtu.be/NfrdH6Asx10}{here}.

\begin{figure}[!b]
    \centering
    \begin{minipage}{.49\textwidth}
    \centering
    \includegraphics[width=0.8\linewidth]{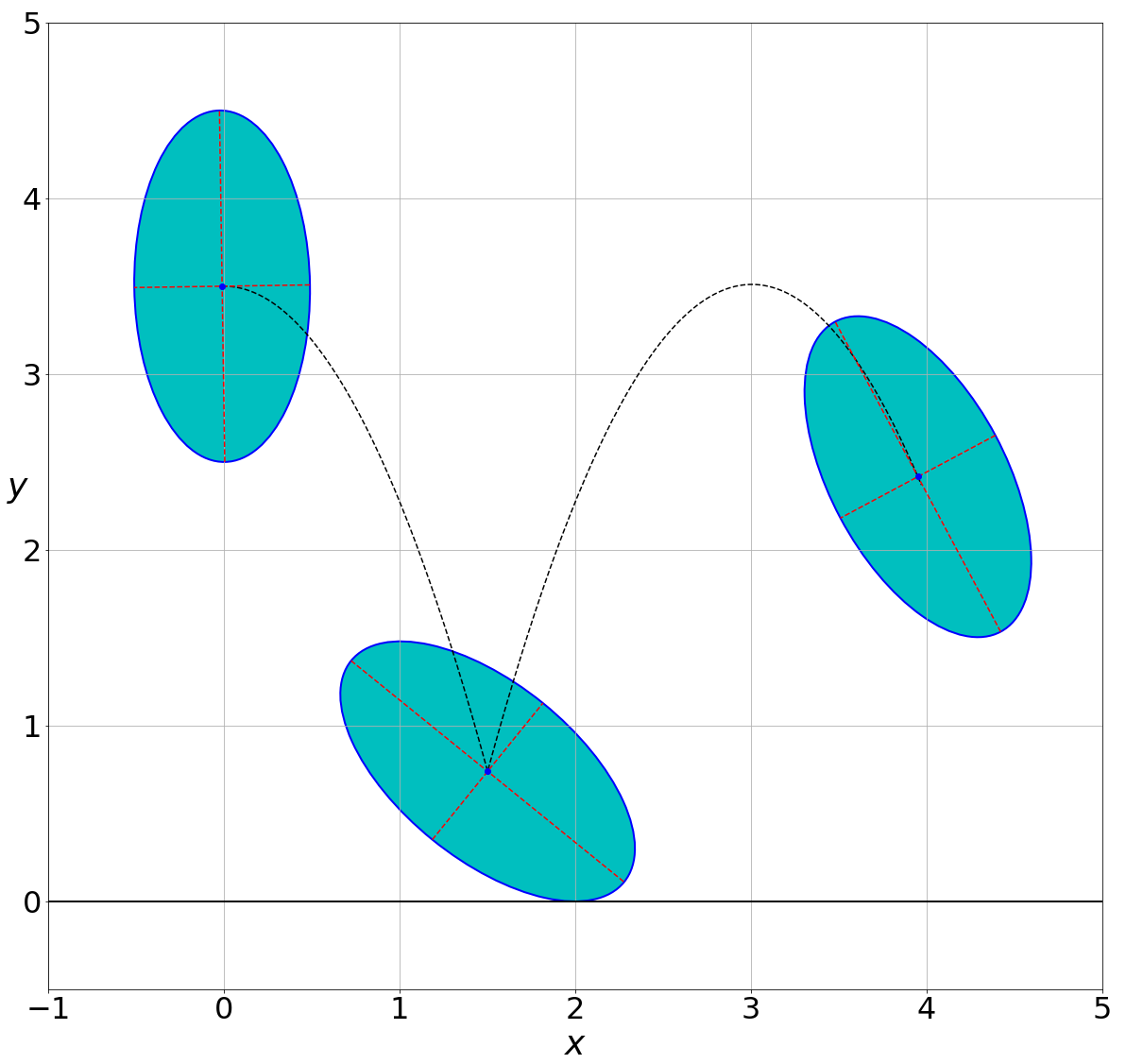}
    \caption{Trajectory of an ellipse rotating and colliding with the floor. The initial and final times are $t_0=0$ and $t_N=2$.}
    \label{fig:trajectory}
    \end{minipage}
    \begin{minipage}{.5\textwidth}
    \centering
    \includegraphics[width=\linewidth]{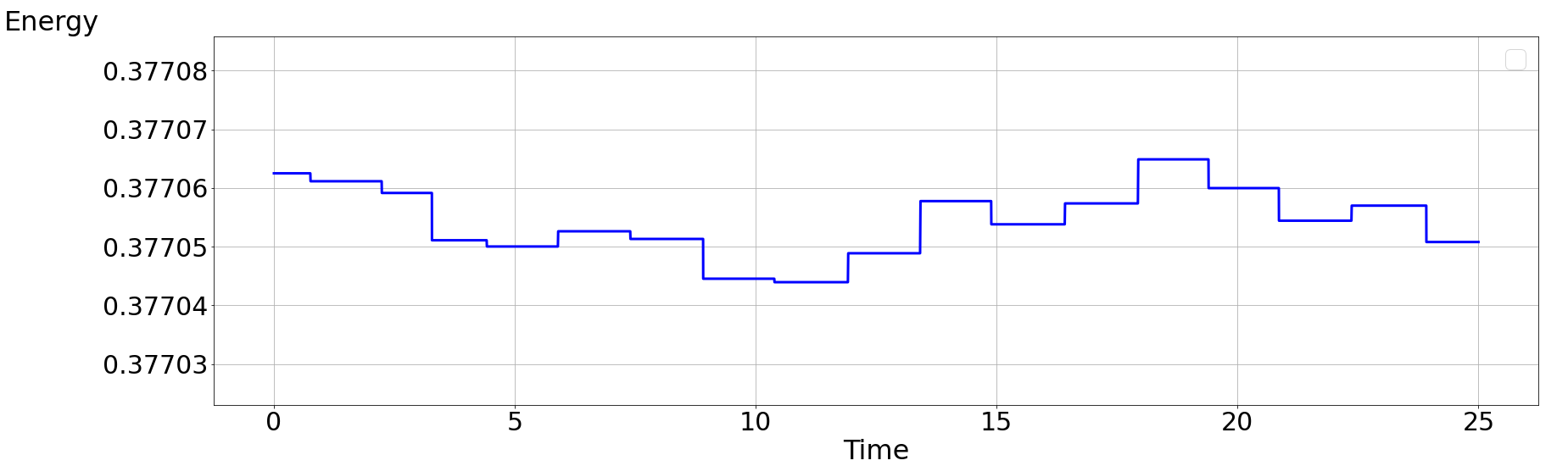}
    \caption{Long-time evolution of the energy of the ellipse colliding with the floor. The initial and final times are $t_0=0$ and $t_N=25$. The system underwent 17 impacts during that period of time.}
    \label{fig:energy}
    \end{minipage}
\end{figure}

\subsection{Nonholonomic spherical pendulum hitting a cylindrical surface}\label{sec:pendulum}

Let us consider a spherical pendulum (cf. \cite[\S 5.1]{EyCoBl2021}) hitting a cylindrical surface. The configuration space of the system is $Q=\mathbb S^2$, where we pick spherical coordinates $q=(\theta,\varphi)\in[0,\pi)\times[0,2\pi)$. The Lagrangian, $L:T\mathbb S^2\to\mathbb R$ consists of the kinetic kinetic energy, together with the gravitational potential,
\begin{equation}\label{eq:lagrangianpendulum}
L(\theta,\varphi;v_\theta,v_\varphi)=\frac{1}{2}m\ell^2(v_\theta^2+v_\varphi^2\sin^2\theta)-mg\ell\cos\theta,
\end{equation}
where $m,\ell,g>0$ are the mass and the length of the pendulum, and the gravitational acceleration, respectively, and we denote $v=v_\theta\,\partial_\theta+v_\varphi\,\partial_\varphi\in T_q \mathbb S^2$. We assume that the polar and azimuthal velocities are proportionally related by a function depending only on the polar angle, i.e., $v_\varphi=f(\theta) v_\theta$ for some $f:\mathbb R\to\mathbb R$ such that $f(\pi)=f(0)$ for each $\theta\in\mathbb R$. Note that, for $f\equiv 0$, we recover a standard 1-dimensional pendulum. This results in the following non-holonomic constraint:
\begin{align*}
\Delta_Q & =\operatorname{span}\{\partial_\theta+f(\theta)\partial_\varphi\},\\\Delta_Q^\circ & =\operatorname{span}\{\omega^1=f(\theta)d\theta-d\varphi\}.
\end{align*}
Due to the pendulum moving inside a cylinder, the admissible set reads
\begin{equation*}
S=\{(\theta,\varphi)\in\mathbb S^2\mid\ell\sin\theta\leq R\},
\end{equation*}
where $R>0$ is the radius of the cylinder. The tangent and cotangent bundles of $\partial S$ are readily seen to be $T\partial S=\operatorname{span}\{\partial_\varphi\}$ and $T^*\partial S={\rm span}\{d\varphi\}$, respectively, and the annihilator of the tangent bundle is given by $(T\partial S)^\circ=\operatorname{span}\{d\theta\}$. Similarly, the tangent map of the boundary inclusion and its adjoint are given by
\begin{align*}
T\imath_S:T\partial S\to T\mathbb S^2|_{\partial S},\quad & (q,v)\mapsto T_q\imath_S(q,v)=v_\varphi\,\partial_\varphi\\
\imath_S^*:T^*\mathbb S^2|_{\partial S}\to T^*\partial S,\quad & (q,p)\mapsto\imath_S^*(q,p)=p_\varphi\,d\varphi,
\end{align*}
for each $v=v_\varphi\,\partial_\varphi\in T_q\partial S$ and $p=p_\theta\,d\theta+p_\varphi\,d\varphi\in T_q^*\mathbb S^2$. Hence, the projection and its adjoint can be chosen as
\begin{align*}
\pi_S:T\mathbb S^2|_{\partial S}\to T\partial S,\quad & (q,v)\mapsto\pi_S(q,v)=v_\varphi\,\partial_\varphi,\\
\pi_S^*:T^*\partial S\to T^*\mathbb S^2|_{\partial S},\quad & (q,p)\mapsto\pi_S^*(q,p)=p_\varphi\,d\varphi,
\end{align*}
for each $v=v_\theta\,\partial_\theta+v_\varphi\,\partial_\varphi\in T_q\mathbb S^2$ and $p=p_\varphi\,d\varphi\in T_q^*\partial S$.

Given a time-step $h\in(0,1)$, the retraction $\mathcal R:T\mathbb S^2\to\mathbb S^2$ may be chosen as
\begin{align*}
\mathcal R_q(v)=(\theta+h\,v_\theta,\varphi+h\,v_\varphi),\qquad(q,v)\in T\mathbb S^2
\end{align*}
For fixed $q=(\theta,\varphi)\in \mathbb S^2$, the inverse of $\mathcal R_q:T_q\mathbb S^2\to\mathbb S^2$ is given by
\begin{align*}
& \mathcal R_q^{-1}(v)=\left(\frac{v_\theta-\theta}{h}\,\partial_\theta+\frac{v_\varphi-\varphi}{h}\,\partial_\varphi\right),\qquad v\in\mathbb S^2.
\end{align*}
As a result, for each $(q,v,h)\in\mathbb S^2\times\mathbb S^2\times(0,1)$, the discrete Lagrangian $L_d:\mathbb S^2\times\mathbb S^2\times(0,1)\to\mathbb R$ and the discrete constraint map $\omega_{d+}^1:\mathbb S^2\times \mathbb S^2\to\mathbb R$ read
\begin{align*}
L_d(q,v,h) & =h\,L\left(q,\mathcal R_q^{-1}(v)\right)=\frac{m\ell^2}{2h}\big((v_\theta-\theta)^2+(v_\varphi-\varphi)^2\sin^2\theta\big)-hmg\ell\cos\theta,\\
\omega_{d+}^1(q,v) & =\omega_q^1(\mathcal R_q^{-1}(v))=\frac{1}{h}\left(f(\theta)(v_\theta-\theta)-v_\varphi+\varphi\right).
\end{align*}

The evolution of the pendulum is simulated by applying Algorithm \ref{alg}. The parameters of the model are chosen as follows: $m=1$, $g=9.8$, $\ell=2$, $R=1.5$ and $f(\theta)=\pi+\cos^2\theta$. Similarly, the initial conditions are chosen as $q(0)=(\theta(0),\varphi(0))=(0.75\pi,0)$ and $v(0)=(v_\theta(0),v_\varphi(0))=(0.25\pi,0.25(\pi+0.5)\pi)$. In Figure \ref{fig:angles}, we depict the evolution of the angles during five seconds. The system underwent three collisions during that period. This \href{https://youtube.com/shorts/3jK2AuNZgLc}{video} shows the motion of the pendulum during that period. Note that the  `spiral' trajectory is due to the nonholonomic constraint. In order to investigate the long-time energy behavior of the integrator, we run the simulation for 100 seconds, yielding the results shown in Figure \ref{fig:energypendulum}. Note that that it remains almost constant with a slight increase within time.

\begin{figure}[!b]
    \centering
    \begin{minipage}{.49\textwidth}
    \centering
    \includegraphics[width=0.8\linewidth]{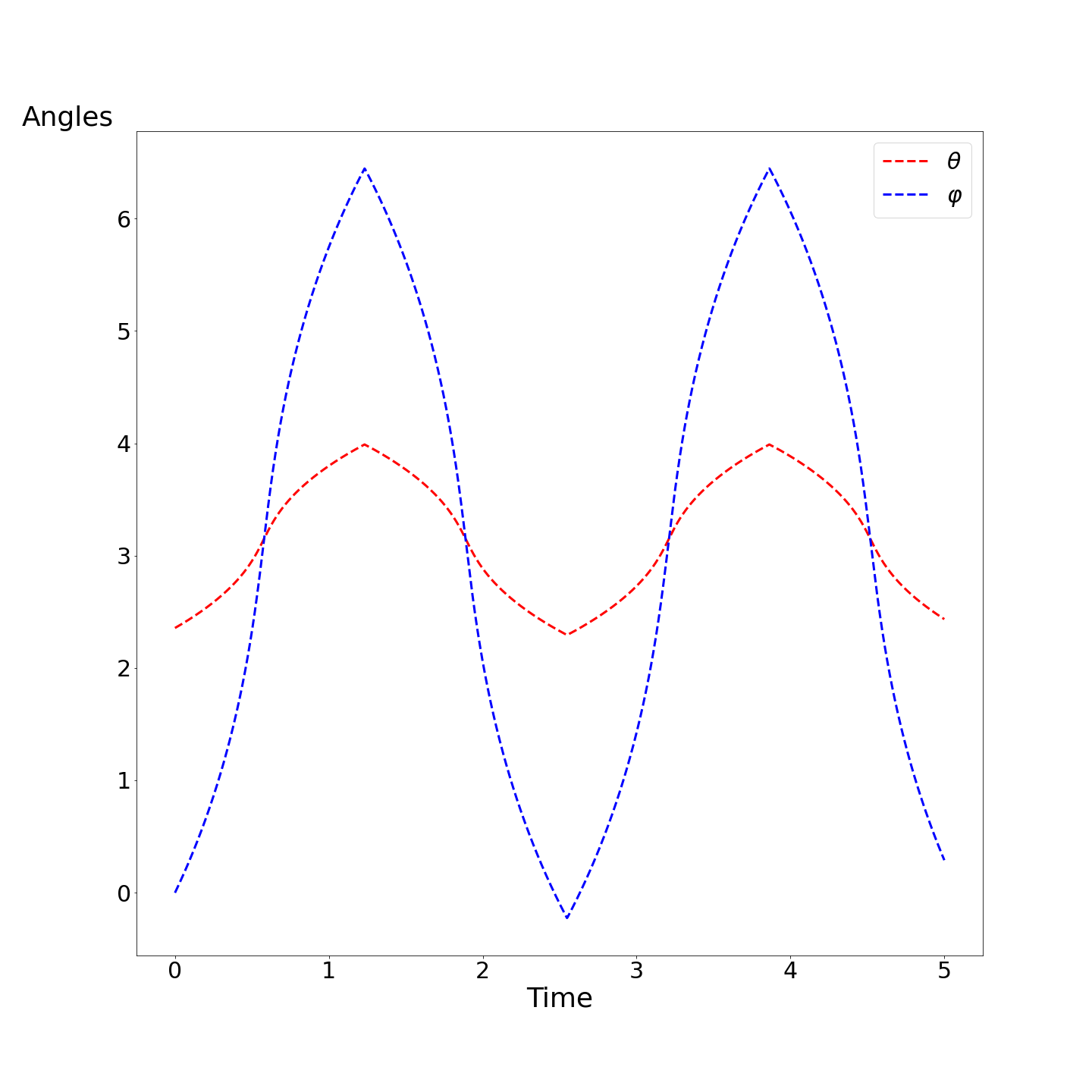}
    \caption{Evolution of the angles of a nonholonomic spherical pendulum colliding with a cylindrical surface. The initial and final times are $t_0=0$ and $t_N=5$, and the time-step is $h=10^{-3}$.}
    \label{fig:angles}
    \end{minipage}
    \begin{minipage}{.5\textwidth}
    \centering
    \includegraphics[width=\linewidth]{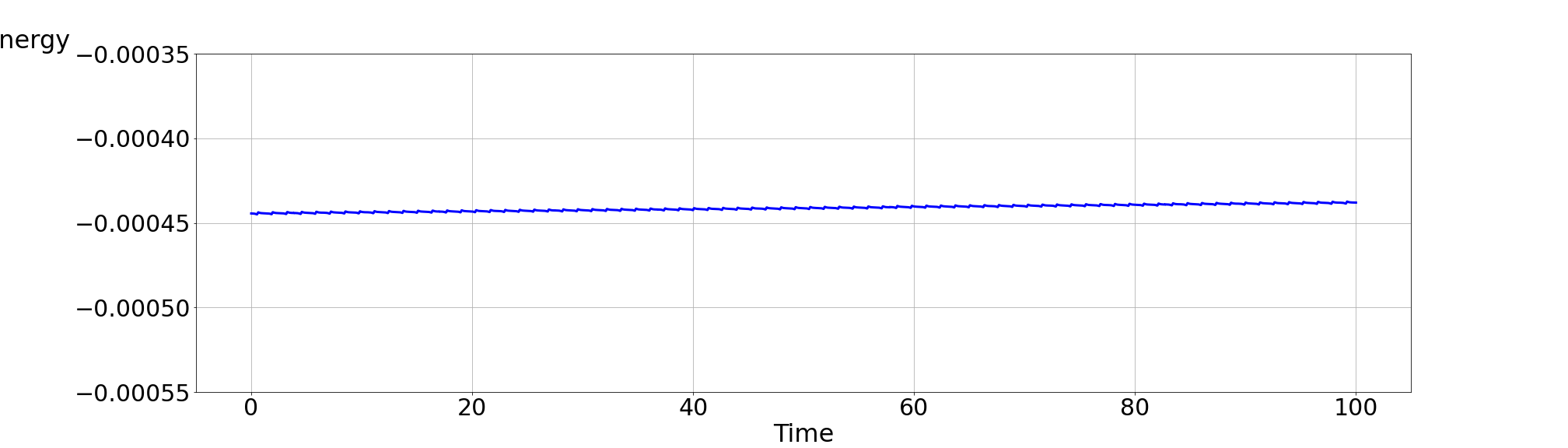}
    \caption{Long-time evolution of the energy of a nonholonomic spherical pendulum colliding with a cylindrical surface. The initial and final times are $t_0=0$ and $t_N=100$, and the time-step is $h=10^{-4}$. The system underwent 76 impacts during that period of time.}
    \label{fig:energypendulum}
    \end{minipage}
\end{figure}

\section{Conclusions and future work}

In this paper, a discrete analogous to the nonsmooth variational approach introduced in \cite{RoCo2023,RoCo2024} to treat implicit nonholonomic Lagrangian systems with collisions has been developed, thus extending the discrete theory for unconstrained systems presented in \cite{FeMaOrWe2003}. After giving the discrete Lagrange--d'Alembert--Pontryagin variational principle, the corresponding equations are computed. They consist of the discrete nonholonomic implicit Euler--Lagrange equations obtained in \cite{LeOh2010,LeOh2011}, together with the discrete conditions for the elastic impact. The forward in time implementation of these equations is summarized in Algorithm \ref{alg}. Lastly, some examples illustrating the theory are presented, including a bouncing ellipse and a nonholonomic spherical pendulum, and several simulations are performed to test the validity of the variational integrators.

For future work, we would like to develop a discrete reduction theory by mimicking \cite{RoCo2023}, thus obtaining reduced variational integrators. The removal of superfluous degrees of freedom will lead to significantly faster simulations. Similarly, we would like to explore interconnection, investigating how the impacts on a subsystem transfer to the rest of them due to the coupling. Furthermore, this interconnection approach will allow for considering systems with open ports that can be utilized to control the system \cite{YoMa2012}. In turn, the controller of the system may be computed from the optimal control perspective \cite{CoMa2013}, i.e., by minimizing some cost functional. In any case, a critical step will be to ensure that the input preserves the constraints on the state, as well as on its velocity and momentum, at the impact time. Lastly, we would like to compare the approach followed here to carry out discretization with the alternative proposals presented in \cite{CaFeToZu2023,PeYo2024}, testing the performance of the integrators obtained from the different perspectives.

\bibliographystyle{siam}
\bibliography{biblio}

\end{document}